\documentclass[11pt]{amsart}
\usepackage{mathrsfs}
\usepackage{amsfonts}
\usepackage{amsmath}
\usepackage{amssymb}
\usepackage{amsthm}
\usepackage{enumerate}
\usepackage{color}
\usepackage{geometry}
\usepackage{hyperref}
\usepackage{float}
\usepackage{graphicx}
\usepackage{multirow}
\usepackage[numbers,sort&compress]{natbib}
\usepackage{color}
\usepackage[all]{xy}
\usepackage{cases}

\allowdisplaybreaks
\numberwithin{equation}{section}
\theoremstyle{plain}
\newtheorem{prop}{Proposition}[section]
\newtheorem{coro}[prop]{Corollary}

\newtheorem{lemm}[prop]{Lemma}

\newtheorem{thm}[prop]{Theorem}

\theoremstyle{definition}
\newtheorem{defi}[prop]{Definition}

\newtheorem{exam}[prop]{Example}
\newtheorem{rema}[prop]{Remark}

\newcommand\HS[1]{\leavevmode\null\hspace{#1mm}}

\newcommand\wdots{, ...\HS{0.2}, }
\newcommand\xx{x}
\newcommand\XXX{X}
\newcommand\yy{y}

\newcounter{ITEM}
\newcommand\ITEM[1]{\setcounter{ITEM}{#1}\leavevmode\hbox{\rm(\roman{ITEM})}}

\title{Free Products of Trialgebras}

\author{Juwei Huang}
\address{J.H., School of Mathematical Sciences, South China Normal University, Guangzhou 510631, P. R. China}
\email{juwei1985@126.com}

\author{Yuqun Chen$^{\dagger}$}
\address{Y.C., School of Mathematical Sciences, South China Normal University, Guangzhou 510631, P. R. China}
\email{yqchen@scnu.edu.cn}

\author{Zerui Zhang$^{*}$}
\address{Z.Z., School of Mathematical Sciences, South China Normal University, Guangzhou 510631, P. R. China}
\email{zeruizhang@scnu.edu.cn}

\thanks{AMS 2020 {Subject Classification.} 17A61, 16S15,   20M10}


\thanks{${}^{\dagger}$ Supported by the NNSF of China (11571121, 12071156)}

\thanks{${}^*$ Supported by the fellowship of China Postdoctoral Science Foundation 2021M691099 and the Young Teacher Research and Cultivation  Foundation of South China Normal University 20KJ02}
\thanks{${}^*$ Corresponding author}

\begin{document}
\begin{abstract}
We apply the method of Gr\"obner-Shirshov bases for replicated algebras developed by Kolesnikov to offer a general approach for constructing free products of trialgebrs (resp. trioids).   In particular, the  open problem of  Zhuchok on constructing free products of trioids is solved.
\end{abstract}
\maketitle

\section{Introduction}\label{s-it}
  Gr\"obner and  Gr\"obner--Shirshov bases theory was first introduced  by Buchberger~\cite{Buchberger65}  for commutative algebras and  independently by Shirshov~\cite{Sh} for Lie algebras.  Then it was  developed for various kind of algebras in the last tens of years under the name of Gr\"obner bases theory or  Gr\"obner--Shirshov bases theory. We refer to the survey~\cite{bokut-chen-survey} for more interesting history. This theory has become an effective computational tool in algebras, especially for those defined by generators and relations.  In particular, if we use generators and relations to construct free products of certain algebras,  trialgebras for instance, this theory turns out to be useful.

Associative dialgebras  (dialgebras, for short) and associative trialgebras (trialgebras, for short) are generalizations of associative algebras by applying certain ``replication procedure" \cite{Ko17}. Kolesnikov \cite{Ko17} established the Gr\"obner-Shirshov bases method for various replicated algebras such as dialgebras, trialgebras and Leibniz algebras.  In particular, he showed that every dialgebra or trialgebra can be embedded into its universal enveloping associative algebra. This fact on embedding turns out to be very useful in constructing free products of dialgebras or trialgebras.

 Dimonoids~\cite{Lo99} and trioids~\cite{Lo04} are generalizations of semigroups.  Trioids have close relationships with Hopf algebras~\cite{NJC}, Leibniz 3-algebras~\cite{JM} and Rota-Baxter operators~\cite{KJ}. Dimonoids and trioids are interesting subjects, for instance, see~\cite{Di1,Di2,Lo99,Lo04, Zhuchok17,Zhuchok15,Zhu150,Ko}.  Loday and  Ronco~\cite{Lo04}  constructed  a free trioid of rank 1, which can be easily generalized to the case of arbitrary rank, see also~\cite{Zhuchok15}. Trioids satisfying some identities are also studied,   Zhuchok \cite{Zhuchok19}  constructed free commutative trioids.  Free products of two dialgebras  were constructed in~\cite{Di1},  free products of arbitrary  dimonoids were constructed in~\cite{Zhuchok13} and free products
 of arbitrary  doppelsemigroups were constructed in~\cite{Zhuchok17-2}. One of the open problems raised by Zhuchok in \cite{Zhuchok19} is how to construct free products of arbitrary  trioids. The aim of this article is to solve this problem.

The paper is organized as follows:  We first recall the Gr\"obner-Shirshov bases theory for associative algebras. Then we recall the embeddings of dialgebras or trialgebras into their universal enveloping associative algebras. Finally, based on this embedding, we construct a linear basis for a free products of arbitrary  trialgebras. As a special case, if all the trialgebras are trioid algebras, then the linear basis we obtained will be a set of normal forms of elements for a free product of trioids.

\section{Gr\"obner-Shirshov bases theory for trialgebras and dialgebras}\label{sec-gsb-ass}
 The main aim of this article is  to offer a general approach for constructing free products of dialgebras or trialgebras or trioids.   The method we applied here is the Gr\"obner-Shirshov bases theory for dialgebras or trialgebras which is established by Kolesnikov~\cite{Ko17}.  Kolesnikov showed that  every dialgebra or trialgebra can be embedded into its universal enveloping  associative algebra~\cite{Ko17}, in particular, this holds  for free products of arbitrary  dialgebras or trialgebras.   In light of this, we shall first construct the universal enveloping associative algebra of the desired free products of certain algebras, which in turn indicates a linear basis of a free product of dialgebras or trialgebras.

The aim of this section is to recall the Gr\"obner-Shirshov bases  theories for associative algebras and trialgebras, respectively.

\subsection{Gr\"obner-Shirshov bases theory for associative algebras}\label{subsec-cd-ass}
Roughly speaking, the method of Gr\"obner-Shirshov bases for various algebras is a special kind of rewriting rule. In general, one first construct a linear basis for the free algebra under consideration, then for the linear basis, one introduce some well-order which is more or less compatible with the operations of the algebra. Finally, one study the leading monomial of elements of an ideal generated by a set.  And the key lies in the conditions under which one can ``reduce" every element into a linear combination of the linear basis of the quotient algebra. Now we recall the Gr\"obner-Shirshov bases for associative algebras.

Let~$X$ be a well-ordered set and let~$k\langle X\rangle$ be the free associative algebra generated by~$X$ over a field~$k$.
Denote by~$X^*$ the free monoid generated by~$X$,  which consists of all words (i.e., sequences) on $X$ including the empty word (denoted by $\varepsilon$). Clearly, $X^*$ forms a linear basis for~$k\langle X\rangle$.

 For every word~$u=x_{i_1}x_{i_2}... x_{i_n}\in X^*$, where the letters~$x_{i_1},...,x_{i_n}$ lie in~$\XXX$, we define  the  \emph{length}~$\ell(u)$ of~$u$ to be~$n$, in particular,  we have~$\ell(\varepsilon)=0$. We recall that the \textit{deg-lex order} (degree-lexicographic order) on $X^*$ is defined as  follows:
for all words~$u=x_{i_1}x_{i_2}... x_{i_m}$, and~$v=x_{j_1}x_{j_2}... x_{j_n}\in X^*$, where $x_{i_l},~x_{j_t}\in X$, we define
\begin{equation*}
u>v \ \ \  \ \mbox{if}\ (\ell(u),x_{i_1},x_{i_2},..., x_{i_m})>(\ell(v),x_{j_1},x_{j_2},..., x_{j_n}) \ \mbox{lexicographically}.
\end{equation*}
In particular, if~$u$ is not the empty word, then  we have~$u>\varepsilon$.
A well order $>$ on $X^*$ is called \textit{monomial} if for all words~$u,v,w_1,w_2\in X^*$, we have
$$
u>v \Rightarrow w_1uw_2>w_1vw_2.
$$
Clearly, the above deg-lex order on~$X^*$  is a monomial order.

For every nonzero polynomial $f=\sum_{i=1}^{n}\alpha_iu_i\in k\langle X\rangle$,  where $0\neq \alpha_i\in k$, $1\leq i\leq n$, $u_1,\dots,u_n\in X^*$ and~$u_1>u_2>...>u_n$,
we call $u_1$  the \textit{leading monomial} of $f$, denoted by $\overline{f}$; and call~$\alpha_1$ the \emph{leading coefficient} of~$f$, denoted by~$\mathsf{lc}(f)$.
A polynomial~$f$ is called \emph{monic} if~$\mathsf{lc}(f)=1$, and a nonempty subset~$S$ of~$k\langle X\rangle$ is called \emph{monic} if every element in $S$ is monic.  For every~$S\subseteq k\langle X\rangle$, define
$$\overline{S}=\{\bar{s}\mid s\in S \}.$$
And for every set~$U,V\subseteq X^*$, define
$$UV=\{uv\mid u\in U, v\in V\}.$$
Denote by~$\mathsf{Id}(S)$ the ideal of $ k\langle X\rangle$ generated by~$S$.  For convenience, we define $\bar 0=0$ and define~$0<u$ for every $u\in X^*$.

A monic set $S$ is called a \emph{Gr\"{o}bner-Shirshov basis} in $ k\langle X\rangle$ for the quotient algebra $k\langle X|S\rangle:= k\langle X\rangle/\mathsf{Id}(S)$, if for every nonzero polynomial~$f\in\mathsf{Id}(S)$ , there exists some words~$u,v\in X^*$  and some element~$s\in S$ satisfying~$ \bar{f}=u\bar{s}v$. By this definition, it is in general not easy to see whether a set~$S$ is a Gr\"{o}bner-Shirshov basis for~$k\langle X|S\rangle$, so in practice, we usually consider the following special elements in the ideal generated by~$S$:  For all~$f,g\in S$,
\ITEM1 If there exist words~$u,v\in X^*$
such that $\bar{f}=u\bar{g}v$, then $(f,g)_{\bar{f}}=f-ugv$ is called an
\textit{inclusion composition} of~$f$ and~$g$ with respect to~$\bar{f}$;
\ITEM2 If there exist words~$u,v\in X^*$
such that $\bar{f}u=v\bar{g}$ and
$1\leq \ell(u)<\ell(\bar{g})$, then $(f, g)_{\bar{f}u}=fu-vg$ is
called  an \textit{intersection composition} of~$f$ and~$g$ with respect to~$\bar{f}u$.
If all the compositions of the form~$(f,g)_w$ can be written as linear combinations of elements of the form~$u_is_iv_i$ satisfying~$u_i\overline{s_i}v_i<w$, where all~$u_i,v_i$ lie in~$X^*$ and each~$s_i$ lies in~$S$, then the composition~$(f,g)_w$ is called trivial.

The following result was implicitly proved in~\cite{Sh}, and independently proved in~\cite{Ber}, see also~\cite{Bo,Bo1}. By this lemma, in order to see whether a set~$S$ forms a Gr\"obner-Shirshov basis for~$k\langle X|S\rangle$, it suffices to calculate all the possible compositions.

\begin{lemm}\cite{Sh,Ber,Bo,Bo1}\label{l1} Let $S$ be a monic subset of~$k\langle X\rangle$  and let~$\mathsf{Id}(S)$ the ideal of $k\langle X\rangle$ generated by $S$. Then the following statements are equivalent.

\ITEM1 All the compositions of elements in~$S$ are trivial.

\ITEM2 The set~$S$ is a Gr\"obner-Shirshov basis for $k\langle X|S\rangle$, that is, for every nonzero polynomial~$f\in \mathsf{Id}(S)$, we have~$\bar{f}=u\bar{s}v$
for some  element~$s\in S$ and for some words~$u,v\in  X^*$.

\ITEM3 $\mathsf{Irr}(S):=\{u+\mathsf{Id}(S)\mid u\in X^*\setminus X^*\overline{S}X^*\}$ forms a $ k$-basis of the quotient associative algebra $ k\langle X|  S\rangle$.
\end{lemm}

Now we recall a known fact on Gr\"obner-Shirshov bases for associative algebras, and we omit the proof since it is quite similar to those for Lemmas~\ref{lemm-gsb} and~\ref{thm-fpt}.
 \begin{lemm}\label{ass-gsb}
   Let~$\{k\langle Y_i | U_i\rangle\mid i\in I\}$ be a family of associative algebras such that $Y_i\cap Y_j =\emptyset$ for all distinct~$i,j\in I$. If for every~$i\in I$, $U_i$ is a Gr\"obner-Shirshov basis for~$k\langle Y_i | U_i\rangle$, then~$\cup_{i\in I}U_i$ forms a Gr\"obner-Shirshov basis for~$k\langle \cup_{i\in I}Y_i | \cup_{i\in I}U_i\rangle$. In particular, for every index~$j\in I$, $k\langle Y_j | U_j\rangle$ is isomorphic to the subalgebra of~$k\langle \cup_{i\in I}Y_i | \cup_{i\in I}U_i\rangle$  generated by $\{y_j+\mathsf{Id}(\cup_{i\in I}U_i)\mid y_j\in Y_j\}$ under the isomorphism which maps~$y_j+\mathsf{Id}(U_j)$ to~$y_j+\mathsf{Id}(\cup_{i\in I}U_i)$ for every~$y_j$ in~$Y_j$.
     \end{lemm}

\begin{rema}
 By Lemma \ref{ass-gsb}, one can easily construct a linear basis of a free product of an arbitrary family of associative algebras. Unfortunately, a similar result of Lemma~\ref{ass-gsb} does not hold for dialgebras or trialgebras,  see  \cite[Theorem 4.8]{Di1} for instance.
 \end{rema}

\subsection{Embeddings of dialgebras and trialgebras into associative algebras}
In this section,  we shall shortly recall the definition of dialgebras, trialgebras, and the embeddings of dialgebras and trialgebras into their corresponding universal enveloping associative algebras. All of the results in this section can be found in~\cite{Ko17}. For the convenience of the readers, we shortly recall what are needed here.

\begin{defi}\cite{Lo04}\label{defi-dia}
A \emph{dialgebra} (resp. \emph{dimonoid}) is a vector space (resp. set)~$(D, \vdash,\dashv)$ equipped with two binary associative operations
$\vdash$ and $\dashv$ such that the following identities hold:
\begin{equation}
\begin{cases}
a \dashv(b \vdash c)=a\dashv (b\dashv c),\\
(a\dashv b)\vdash c=(a\vdash b)\vdash c, \\
a\vdash(b\dashv  c)=(a\vdash b)\dashv c,
\end{cases}
\end{equation}
for all~$a, b$ and~$c$ in~$D$.    We denote by~$\mathsf{V}^{(2)}(X)$ (resp. $\mathsf{Vid}^{(2)}(X)$) the free dialgebra (resp. dimonoid) generated by~$X$.
\end{defi}

\begin{defi}\cite{Lo04}\label{l2}
A \emph{trialgebra} (resp. \emph{trioid}) is a vector space (resp. set) $(T,\vdash,\dashv,\perp)$ equipped with 3 binary associative operations: $\vdash$, $\dashv$ and $\perp$, satisfying the following identities:
\begin{equation}\label{eq00}
\begin{cases}
a\dashv(b\vdash c)=a\dashv (b\dashv c),\\
(a\dashv b)\vdash c=(a\vdash b)\vdash c, \\
a\vdash(b\dashv c)=(a\vdash b)\dashv c,\\
a\dashv(b\perp c)=a\dashv (b\dashv c),\\
(a\perp b)\vdash c=(a\vdash b)\vdash c, \\
a\vdash(b\perp c)=(a\vdash b)\perp c,\\
a\perp(b\dashv c)=(a\perp b)\dashv c,\\
a\perp(b\vdash c)=(a\dashv b)\perp c
\end{cases}
\end{equation}
for all $a, \ b, \ c\in T$.   We denote by~$\mathsf{V}^{(3)}(X)$ (resp. $\mathsf{Vid}^{(3)}(X)$) the free trialgebra (resp. trioid) generated by~$X$.
\end{defi}

Obviously, every trialgebra can be represented by generators and relations: For a trialgebra $T$, if $X=\{a_i\mid i\in I\}$ is a $k$-basis of $T$ and
$$S=\{a_i\delta a_j=\{a_i\delta a_j\}\mid i,j\in I,\delta\in\{\vdash, \dashv, \perp\}\}$$ is the multiplication table, where~$\{a_i\delta a_j\}$ is of the form~$\Sigma \alpha_{_{\delta ,ijt}}a_t$, $\alpha_{_{\delta ,ijt}}\in k$ and $a_t\in X$,  then we have
 $$T\cong \mathsf{V}^{(3)}(X|S):=\mathsf{V}^{(3)}(X)/\mathsf{Id}^{(3)}(S),$$
 where~$\mathsf{Id}^{(3)}(S)$ is the ideal of~$ \mathsf{V}^{(3)}(X)$ generated by~$S$ .

Similarly, let~$T$ be a trioid. For~$\delta\in\{\vdash, \dashv, \perp\}$, denote by~$\{x\delta y\}$ the letter in~$T$ that is equal to~$x\delta y$.  Let~$\mathsf{Vid}^{(3)}(T)$ be the free trioid generated by~$T$ and denote by~$\rho_{_S}$ the congruence of~$\mathsf{Vid}^{(3)}(T)$ generated by
 $$
 S=\{(x\delta y,\{x\delta y\})\mid x,y\in T, \delta\in\{\vdash, \dashv, \perp\}\}.
 $$
 Then it is clear that we have
 $$T\cong \mathsf{Vid}^{(3)}(T|S):=\mathsf{Vid}^{(3)}(T)/\rho_{_S}.$$
 Similar results hold for dialgebras and dimonoids.    In particular, we can denote by~$\mathsf{V}^{(t)}(X|S) $ (resp. $\mathsf{Vid}^{(t)}(X|S)$) an arbitrary trialgebra (resp. trioid) when~$t=3$ and an arbitrary dialgebra (resp. dimonoid) when~$t=2$.

Note that if $(T, \vdash,\dashv,\perp)$ is a trioid (resp. trialgebra), then $(T,\vdash,\dashv)$ is a dimonoid (resp. dialgebra). Obviously, dialgebras and trialgebras are generalizations of associative algebras.
Recall that for every dialgebra (dimonoid)~$D$, for all~$\yy_1\wdots\yy_m$ in~$D$, by~\cite{Lo99} every parenthesizing of
$$
\yy_1\vdash\cdots\vdash\yy_t\dashv\cdots\dashv\yy_m
$$
 gives the same element in~$D$,  which we denote by~$\yy_1...\dot{\yy}_t...\yy_m$.    For instance, we have~$\yy_3\yy_2\dot{\yy}_1=\yy_3\vdash\yy_2\vdash\yy_1$ and~$\dot{\yy}_7\yy_8=\yy_7\dashv\yy_8$.

 Similarly, let~$T$ be a trialgebra (trioid). Then for all~$\yy_1,...,\yy_m\in T$,  by~~\cite{Lo04} every parenthesizing of
$$
(\yy_1\vdash\cdots \vdash \yy_{m_{1}-1})\vdash(\yy_{m_1}\dashv\cdots \dashv \yy_{m_{2}-1})\perp(\yy_{m_2}\dashv\cdots \dashv \yy_{m_{3}-1})\perp\cdots \perp(\yy_{m_{r}}\dashv\cdots\dashv \yy_{m})
$$
gives the same element in $T$, which we denote by
$$ \yy_1\dots\yy_{m_{1}-1}\dot{\yy}_{m_1}\yy_{m_1+1}\dots \yy_{m_{2}-1}\dot{\yy}_{m_2}\yy_{m_{2}+1}\dots \yy_{m_{r}-1}\dot{\yy}_{m_{r}}\yy_{m_{r}+1}\dots\yy_{m}.$$

These notations indicate some connections between dialgebras (or trialgebras) with associative algebras. More precisely, Kolesnikov~\cite{Ko17} proved a  surprising result that every trialgebra (resp. dialgebra) generated by~$X$ can be embedded into an associative algebra generated by~$X\cup \dot{X}$, where~$\dot{X}=\{\dot{x}\mid \xx\in X\}$ is a copy of~$X$, namely for every~$x$ in~$X$, the notation~$\dot{x}$ means a letter corresponding to~$x$.  Now we shall explain the embedding in more details.

As in Section~\ref{subsec-cd-ass}, denote by~$k\langle X\cup\dot{X}\rangle$ and~$k\langle X\rangle$ the free associative algebras generated by~$X\cup\dot{X}$ and~$X$ respectively. Obviously, there exists a unique homomorphism
\begin{equation}\label{eq-varphi}
\varphi: k\langle X\cup\dot{X}\rangle \rightarrow k\langle X\rangle
 \end{equation}
 induced by $x\mapsto x$ and~$\dot{x}\mapsto x$ for all~$x \in X$.  A more precise notation for~$\varphi$ might be~$\varphi_{_X}$, but to simplify the formulas we do not add the subscript because each time the underlying generating set is clear. In particualr, sometimes we do not denote the generating set by~$X$, but~$\varphi$ always means the homomorphism by removing the dot.

  If we define three binary operations~$\vdash$, $\dashv$ and~$\perp$  in~$k\langle X\cup\dot{X}\rangle $ as follows:
\begin{equation}\label{eq1}
a \vdash b=\varphi(a)b,\  a\dashv b=a\varphi(b),\  a\perp b=ab,
\end{equation}
 for all~$a,b\in k\langle X\cup\dot{X}\rangle$,  then by~\cite{Ko17},  we obtain that
 $$
 k\langle X\cup\dot{X}\rangle^{(3)}:=(k\langle X\cup\dot{X}\rangle, \vdash, \dashv, \perp)
 $$
 becomes a trialgebra and
 $$
 k\langle X\cup\dot{X}\rangle^{(2)}:=(k\langle X\cup\dot{X}\rangle, \vdash, \dashv)
 $$
 becomes a dialgebra.  Moreover, we can also introduce trialgebras or dialgebras structure on certain associative algebras defined by generators and relations. More precisely,
let~$R$ be a $\varphi$-invariant subset of~$k\langle X\cup \dot{X}\rangle$, namely, $\varphi(R)\subseteq R$.   Then  we define the three binary operations~$\vdash$, $\dashv$ and~$\perp$  in~$k\langle X\cup\dot{X}|R\rangle $ as follows:
\begin{equation} \label{otalg}
\begin{cases}
(a+\mathsf{Id}(R)) \vdash (b+\mathsf{Id}(R))=\varphi(a)b+\mathsf{Id}(R),\\  (a+\mathsf{Id}(R))\dashv(b+\mathsf{Id}(R))=a\varphi(b)+\mathsf{Id}(R),\\
(a+\mathsf{Id}(R))\perp (b+\mathsf{Id}(R))=ab+\mathsf{Id}(R),
\end{cases}
\end{equation}
for all~$a,b\in k\langle X\cup\dot{X}\rangle$.   It is straightforward to show that
\begin{equation} \label{talg}
  k\langle X\cup \dot{X}|R \rangle^{(3)}:=(k\langle X\cup \dot{X}|R\rangle, \vdash, \dashv,\perp)
\end{equation}
is a trialgebra, and
\begin{equation} \label{dalg}
 k\langle X\cup \dot{X}|R \rangle^{(2)}:=(k\langle X\cup \dot{X}|R\rangle, \vdash, \dashv)
 \end{equation}
  is a dialgebra.   Moreover, for~$t=2,3$, we have~$k\langle X\cup \dot{X}|R \rangle^{(t)}=k\langle X\cup \dot{X}|R \rangle$ as a vector space, so they have the same  linear basis.

\begin{rema}\label{psi}
Since~$\mathsf{V}^{(3)}(X)$ (resp. $\mathsf{V}^{(2)}(X)$) is  the free trialgebra (resp. dialgebra) generated by~$X$, there are algebra homomorphisms~$\Psi_3$ and~$\Psi_2$ as follows:
$$\Psi_3  : \mathsf{V}^{(3)}(X) \longrightarrow k\langle X\cup \dot{X}\rangle^{(3)},\  x\mapsto \dot{x},\ x\in X$$
and
$$\Psi_2 : \mathsf{V}^{(2)}(X) \longrightarrow k\langle X\cup \dot{X}\rangle^{(2)},\  x\mapsto \dot{x},\ x\in X.$$
To simplify the formulas, we denote them by
$$\Psi_{t}  : \mathsf{V}^{(t)}(X) \longrightarrow k\langle X\cup \dot{X}\rangle^{(t)},\  x\mapsto \dot{x},\ x\in X$$ for~$t=2,3$. By~\cite{Ko17}, we know that~$\Psi_{t} $ is an embedding of algebras. Moreover,  for every element~$f\in \Psi_{t}(\mathsf{V}^{(t)}(X))$, we denote by $\Psi_{t}^{-1}(f)$ the pre-image of~$f$. Again, a more precise notation for~$\Psi_t$ could be~$\Psi_{(t,X)}$, but to simplify the formulas, we do not add the subscript~$X$. When the generating set is not denoted by~$X$, the maps~$\Psi_2$ and~$\Psi_3$ are defined in the same way.
\end{rema}

\begin{exam}
Let~$X=\{x,y\}$ and let~$f=x\vdash(x\dashv y)\perp(y\dashv x)-x\vdash x\in\mathsf{V}^{(3)}(X)$. Then we have $\dot{X}=\{\dot{x}, \dot{y}\}$,
$$\Psi_{3}(f)=x\dot{x}y\dot{y}x-x\dot{x}\in k\langle X\cup \dot{X}\rangle$$ and
$$\varphi(\Psi_{3}(f))=xxyyx-xx\in k\langle X\rangle.$$ On the other hand, for the element~$\dot{y}xy\dot{x}x$ in~$\Psi_{3}(\mathsf{V}^{(3)}(X))$, we have
$$\Psi_{3}^{-1}(\dot{y}xy\dot{x}x)=(y\dashv x\dashv y)\perp(x\dashv x).$$
\end{exam}

Let~$S$ be a subset of~$V^{(t)}(X)$ (resp. $k\langle X\rangle$) and denote by~$\mathsf{Id}^{(t)}(S)$ (resp. $\mathsf{Id}(S)$)  the ideal of $\mathsf{V}^{(t)}(X)$ (resp. $k\langle X\rangle$) generated by $S$.   Now we recall one of the main theorems in~\cite{Ko17}, which establishes the Gr\"obner-Shirshov bases theory for trialgebras and dialgebras, respectively.

\begin{lemm} \emph{\cite[Corollary 4.3]{Ko17}} \label{emb} For~$t=2$ or $3$, the algebra
 $V^{(t)}(X|S)$ is isomorphic to the subalgebra of
$k\langle X\cup \dot{X}|\Psi_t(S)\cup \varphi(\Psi_t(S)) \rangle^{(t)}$ generated by~$\{\dot{x}+\mathsf{Id}(\Psi_{t}(S)\cup \varphi(\Psi_{t}(S)))\mid \dot{x}\in \dot{X}\}$ under the embedding
\begin{align*}
\widetilde{\Psi}_{t}: \ &\mathsf{V}^{(t)}(X|S) \mapsto k\langle X\cup \dot{X}|\Psi_{t}(S)\cup \varphi(\Psi_{t}(S)) \rangle^{(t)}, &\\ &x+\mathsf{Id}^{(t)}(S) \mapsto \dot{x}+\mathsf{Id}(\Psi_{t}(S)\cup \varphi(\Psi_{t}(S)))&
\end{align*}
for every~$x\in X$.   Moreover, if some set~$Y^{(t)} \subseteq (X\cup \dot{X})^*$ consisting of monomials satisfies that
$$\{y+\mathsf{Id}(\Psi_{t}(S)\cup \varphi(\Psi_{t}(S)))\mid y\in Y^{(t)}\}$$
 forms a linear basis for $k\langle X\cup \dot{X}|\Psi_{t}(S)\cup \varphi(\Psi_{t}(S)) \rangle^{(t)}$, then
 $$\{\Psi_{t}^{-1}(u)+\mathsf{Id}^{(t)}(S) \mid u\in Y^{(t)}\cap \Psi_{t}(\mathsf{V}^{(t)}(X))\}$$
 forms a linear basis for~$\mathsf{V}^{(t)}(X|S)$.
\end{lemm}

\section{Free products of trialgebras and trioids}\label{s-pro}
This section contains the main result of this article.  We shall first construct free products of arbitrary  trialgebras or dialgebras. And if all the trialgebras (resp. dialgebras) are trioid algebras (resp. dimonoid algebras), then a linear basis of the corresponding free product consisting of monomials forms a set of normal forms for the corresponding free products of the trioids (resp. dimonoids).

We first recall the definition of free products of trialgebras (resp. dialgebras, trioids, dimonoids).
\begin{defi}
Let $\{T_i \mid i\in I\}$ be a family of indexed trialgebras (resp. dialgebras, trioids, dimonoids), and let $T$ be a trialgebra  (resp. dialgebra, trioid, dimonoid) such that\\
\ITEM1 there exists an injective trialgebra  (resp. dialgebra, trioid, dimonoid) homomorphism $\tau_i: T_i\longrightarrow T$ for every~$i\in I$;\\
\ITEM2 if $T'$ is a trialgebra  (resp. dialgebra, trioid, dimonoid), and if there exists a trialgebra  (resp. dialgebra, trioid, dimonoid) homomorphism $\phi_i: T_i\rightarrow T'$ for every~$i\in I$, then there exists a unique trialgebra  (resp. dialgebra, trioid, dimonoid)  homomorphism $\eta: T\rightarrow T'$ such that the diagram
$$
\xymatrix{
  T_i \ar[rr]^{\tau_i} \ar[dr]_{\phi_i}
                &  &    T \ar[dl]^{\eta}    \\
                & T'                 }
$$
commutes for every~$i\in I$.

Then $T$ is called a \textit{free product} of $\{T_i\mid i\in I\}$.
\end{defi}

Obviously, the free product of  the family~$\{T_i\mid i\in I\}$ of trialgebras  (resp. dialgebras, trioids, dimonoids) is unique up to isomorphism.

Let $\{T_i\mid i\in I\}$ be a family of trialgebras (resp. dialgebras, trioids, dimonoids).
If for every~$i\in I$, $T_i$ is isomorphic to~$T_i'$, then the free product of~$\{T_i'\mid i\in I\}$ is isomorphic to that of~$\{T_i\mid i\in I\}$. Therefore, when one constructs the free product of arbitrary trialgebras (resp. dialgebras)~$\{T_i\mid i\in I\}$, one may always assume that~$T_i\cap T_j=\{0\}$ for all~$i\neq j\in I$. Similarly, when one constructs the free product of  trioids (resp. dimonids)~$\{T_i\mid i\in I\}$, one can always assume~$T_i\cap T_j=\emptyset$ for all~$i\neq j\in I$.

 Let $T=(T,\vdash, \dashv, \perp)$ be an arbitrary trioid and  let $0$ be a new letter that does not lie in~$T$. We define a trioid $T^0=(T\cup \{0\},\vdash, \dashv, \perp)$ as follows: for all $x,y\in T$ and~$\delta\in\{\vdash, \dashv, \perp\}$, define
 $$
x\delta y=
 \begin{cases}
\{x\delta y\},  \ \ \text{ if } x,y \text{ lie in } T,\\
0 \ \ \ \ \ \ \ \ \text{otherwise.}
\end{cases}
 $$
Then it is clear that $T^0$ is a trioid  and~$T$ is a subtrioid of~$T^0$.

 Now we show that one can easily construct free products of arbitrary  trioids by generators and relations.
\begin{lemm}\label{thm-fpt}
 Let~$\{X_i\mid i\in I\}$ be a pairwise disjoint set   and let $\{\mathsf{Vid}^{(3)}(X_i| S_i)\mid i\in I\}$ be a family of trioids, where~$I$ is an index set. Then  $\mathsf{Vid}^{(3)}(\cup_{i\in I}X_i|\cup_{i\in I}S_i)$  is the free product of trioids~$\{\mathsf{Vid}^{(3)}(X_i| S_i)\mid i\in I\}$.
\end{lemm}
\begin{proof}
 For every letter~$x\in \cup_{i\in I}X_i$,  define~$\mathsf{ind}(x)=i$ if~$x$ lies in~$X_i$.
Let~$T$  be an arbitrary trioid such that there exists a family of homomorphisms  $
\phi_i$ from $\mathsf{Vid}^{(3)}(X_i| S_i)$ to $T$, $i\in I$. Denote by~$\rho$ the congruence of~$\mathsf{Vid}^{(3)}(\cup_{i\in I}X_i)$ generated by~$\cup_{i\in I}S_i$, and denote by~$\rho_i$ the congruence of~$\mathsf{Vid}^{(3)}(X_i)$ generated by~$S_i$.
 Define a map
$$\theta: \cup_{i\in I}X_i \longrightarrow T,  \ x\mapsto \phi_{\mathsf{ind}(x)}(x\rho_{\mathsf{ind}(x)})
$$
 for every~$x\in \cup_{i\in I}X_i$.
Since~$\mathsf{Vid}^{(3)}(\cup_{i\in I}X_i)$ is the free trioid generated by~$\cup_{i\in I}X_i$, there exists a  homomorphism
$$\widetilde{\theta}: \mathsf{Vid}^{(3)}(\cup_{i\in I}X_i) \longrightarrow T, \ x\mapsto \phi_{\mathsf{ind}(x)}(x\rho_{\mathsf{ind}(x)})  $$
 for every~$x\in \cup_{i\in I}X_i$.
Clearly, for every~$(a,b)\in \cup_{i\in I}S_i$, there exists some index~$i\in I$ such that~$(a,b)\in S_i$. So we have~$\phi_i(a\rho_i)=\phi_i(b\rho_i)$, which means that~$\widetilde{\theta}(a)=\widetilde{\theta}(b)$. Thus, there exists a homomorphism
$$\eta: \mathsf{Vid}^{(3)}(\cup_{i\in I}X_i|\cup_{i\in I}S_i) \longrightarrow T, \ x\rho\mapsto \phi_{\mathsf{ind}(x)}(x\rho_{\mathsf{ind}(x)}) $$
 for every~$x\in \cup_{i\in I}X_i$.

Define~$\tau_i$ to be the  homomorphism induced by
$$
\tau_i: \mathsf{Vid}^{(3)}(X_i| S_i) \longrightarrow \mathsf{Vid}^{(3)}(\cup_{i\in I}X_i|\cup_{i\in I}S_i), \ x\rho_i \mapsto x\rho,\ \ \ x\in X_i.
$$
Note that for every $i\in I$, we have~$\rho_i\subseteq\rho$.  So $\tau_i$ is well-defined
  and we have~$\eta \tau_i=\phi_i$ for every~$i\in I$. Such a homomorphism $\eta$ is clearly unique.

   It remains to show that~$\tau_i$ is injective for every~$i\in I$.
For every~$j\in I$, we define a homomorphism~$\phi_j'$ as follows:
$$
\phi_j': \mathsf{Vid}^{(3)}(X_j| S_j)\rightarrow T^0_i, \ x\rho_j\mapsto  x\rho_i \mbox{ if } j=i; \mbox{ and } x\rho_j \mapsto  0 \mbox{ if } j\neq i
 $$
 for every~$x\in X_j$.   Clearly, $\phi_i'$ is injective. Then by the above reasoning,  there exists a homomorphism $\eta'$ such that~$\eta' \tau_j=\phi_j'$ for all $j\in I$. In particular, we have~$\eta' \tau_i=\phi_i'$ and thus $\tau_i$ is injective.
\end{proof}

A similar result holds for dimonoids, dialgebras and trialgebras. We formulate it for trialgebras and dialgebras without proof.

\begin{lemm}\label{thm-fp}
  Let~$\{X_i\mid i\in I\}$ be a pairwise disjoint set  and let $\{\mathsf{V}^{(3)}(X_i|S_i)\mid i\in I \}$  $($resp. $\{\mathsf{V}^{(2)}(X_i|S_i)\mid i\in I \})$ be a family of trialgebras $($resp. dialgebras$)$, where~$I$ is an index set.
 Then~$\mathsf{V}^{(t)}(\cup_{i\in I}X_i|\cup_{i\in I}S_i )$  is the free product of $\{\mathsf{V}^{(t)}(X_i|S_i)\mid i\in I\}$, $t\in \{2,3\}$.  In particular, for every~$i\in I$,  the algebra~$\mathsf{V}^{(t)}(X_i|S_i)$ embeds into~$\mathsf{V}^{(t)}(\cup_{i\in I}X_i|\cup_{i\in I}S_i )$.
\end{lemm}

  By Lemmas~\ref{thm-fpt}, \ref{thm-fp} and~\ref{emb}, in order to construct a set of normal forms of elements for a free product of trialgebras (resp. dialgebras, dimonoids, trioids), we can apply the method of Gr\"obner-Shirshov bases theory for trialgebras (resp. dialgebras, dimonoids, trioids).

Now  we  recall a well order on the free monoid~$(\cup_{i\in I}(X_i\cup\dot{X_i}))^*$ generated by~$(\cup_{i\in I}(X_i\cup\dot{X_i}))$.
Assume that each~$X_i$ is a well-ordered set for every $i\in I$ and assume that~$I$ is well-ordered. Now we extend the well orders on the sets $X_i$ to~$(\cup_{i\in I}(X_i\cup\dot{X_i}))^+$ as follows: for all~$x,y\in\cup_{i\in I}X_i$, define

   \ITEM1 ~$x>y$ if $x\in X_i$,  $y\in X_j$, $i,j\in I$ with $i> j$;

  \ITEM2 ~$\dot{x}> \dot{y}$ if~$x> y$;

  \ITEM3   $\dot{x}>y$. \\
  Then we obtain a well order on $\cup_{i\in I}(X_i\cup\dot{X_i})$.  Finally, let  $>$ be the degree-lexicographic order recalled in Section~\ref{subsec-cd-ass} on~$(\cup_{i\in I}(X_i\cup\dot{X_i}))^{\ast}$.

To avoid too many repetitions on the notations, we shall fix some notations as follows.

\begin{rema}\label{nota-relation}
 Let~$\{X_i\mid i\in I\}$ be a pairwise disjoint set, and let $\{\mathsf{V}^{(t)}(X_i|S_i)\mid i\in I \}$ be a family of dialgebras when~$t=2$ (resp. trialgebras when~$t=3$) such that    $\{x+\mathsf{Id}^{(t)}(S_i)\mid x\in X_i\}$  is a linear basis for~$\mathsf{V}^{(t)}(X_i|S_i)$ and $S_i$ is the multiplication table for~$\mathsf{V}^{(t)}(X_i|S_i)$.    For every letter~$x\in \cup_{i\in I}(X_i\cup \dot{X}_i)$,  define~$\mathsf{ind}(x)=i$ if~$x$ lies in~$X_i\cup \dot{X}_i$.   Recall that~$\varphi$ is defined by the rule in~\eqref{eq-varphi} and~$\Psi_t$,  $t\in \{2,3\}$ is defined by the rule  in Remark~\ref{psi}. (Note that in this situation the generating sets are not denoted by~$X$, but hopefully no confusions will arise.) Define
 $$\mathcal{B}_{i,t}=k\langle X_i|\varphi(\Psi_t(S_i))\rangle,$$
$$\mathcal{A}_{i,t}=k\langle X_i\cup\dot{X_i}|\Psi_t(S_i)\cup \varphi(\Psi_t(S_i)) \rangle \rangle,$$
and $$\mathcal{A}_t=k\langle\cup_{i\in I}(X_i\cup\dot{X_i})|\cup_{i\in I}(\Psi_t(S_i)\cup \varphi(\Psi_t(S_i))) \rangle. $$
Finally, suppose that~$R_{i,t}$ is a Gr\"obner-Shirshov basis for~$\mathcal{B}_{i,t}$ with respect to the degree-lexicographic order~$>$ on $X_i^*$.   Without loss of generality, we may assume $\varphi(\Psi_t(S_i))\subseteq R_{i,t}$.  Define $X_i'=X_i^*\setminus X_i^*\overline{R_{i,t}}X_i^*$. Then we have~$X_i'=X_i\setminus\{\bar{f}\mid f\in R_{i,t}\}$. By Lemma~\ref{l1}, the set~$ \mathsf{Irr}(R_{i,t})=\{x+\mathsf{Id}(R_{i,t})\mid x\in X_i'\}$ forms a linear basis for~$\mathcal{B}_{i,t}$.
\end{rema}

\begin{lemm}\label{lemm-gsb}
  With the notations in Remark~\ref{nota-relation}, the set~$\Psi_t(S_i)\cup R_{i,t}$ forms a Gr\"obner-Shirshov basis for~$\mathcal{A}_{i,t}$.
\end{lemm}
\begin{proof}
Since~$R_{i,t}$ is a Gr\"obner-Shirshov basis for~$\mathcal{B}_{i,t}=k\langle X_i|\varphi(\Psi_t(S_i))\rangle$,  the sets~$R_{i,t}$ and $\varphi(\Psi_t(S_i))$ generate the same ideal of~$k\langle X_i\rangle$. It follows immediately that we have
$$\varphi(\Psi_t(S_i))\subseteq k\langle X_i\rangle R_{i,t} k\langle X_i\rangle\subseteq
k\langle X_i\cup\dot{X_i}\rangle R_{i,t} k\langle X_i\cup\dot{X_i}\rangle$$ and
$$R_{i,t}\subseteq
k\langle X_i\cup\dot{X_i}\rangle \varphi(\Psi_t(S_i))k\langle X_i\cup\dot{X_i}\rangle. $$
So we obtain
  $$\mathcal{A}_{i,t}=k\langle X_i\cup\dot{X_i}|\Psi_t(S_i)\cup \varphi(\Psi_t(S_i)) \rangle \rangle= k\langle X_i\cup\dot{X_i}|\Psi_t(S_i)\cup R_{i,t} \rangle.$$
For all~$f,g\in \Psi_t(S_i)\cup R_{i,t}$,  we shall show that all the possible compositions of the form~$(f,g)_w$ for some~$w\in (\dot{X_i}\cup X_i)^{\ast}$ (if any) are trivial, and then the lemma follows.

First of all, if no  letters   in~$\dot{X_i}$ are involved in the composition~$(f,g)_w$, then it follows that~$f,g$ lie in~$R_{i,t}$. Since~$R_{i,t}$ is a Gr\"obner-Shirshov basis for~$\mathcal{B}_{i,t}$,  it follows that the composition~$(f,g)_w$ can be written as a linear combination of elements of the form~$uhv$ satisfying~$u\overline{h}v<w$, where all~$u,v$ lie in~$(X_i)^*\subseteq (X_i\cup\dot{X_i})^*$ and each~$h$ lies in~$R_{i,t}\subseteq \Psi_t(S_i)\cup R_{i,t}$. Therefore, the composition~$(f,g)_w$ is trivial.

On the other hand, if some letter in~$\dot{X_i}$ is involved in the compostion~$(f,g)_w$, then $f$ or~$g$ lies in~$\Psi_t(S_i)$.  And by the definition of the composition~$(f,g)_w$, every monomial that appears in~$(f,g)_w$ (with nonzero coefficient) must involves some letter in~$\dot{X_i}$. Suppose that~$(f,g)_w=\sum_{n}\alpha_nw_n$ with~$w>\overline{(f,g)_w}=w_1>w_2>\cdots$.
Since~$S_i$ is the multiplication table for~$\mathsf{V}^{(t)}(X_i|S_i)$, we know that every element in~$\Psi_t(S_i)$ is of the form~$y_py_q=\sum_{m}\beta_{m}z_m$, where~$y_p,y_q$ lie in~$X_i\cup \dot{X_i}$ and at least one of~$y_p,y_q$ lies in~$\dot{X_i}$, moreover, every~$z_m$ lies in~$\dot{X_i}$. It follows that we can rewrite~$(f,g)_w$ into the form
$$\sum_{l}\gamma_lu_lf_lv_l+\sum_n\delta_nz_n',$$ where each~$z_n'$ lies in~$\dot{X_i}$, each~$f_l$ lies in~$\Psi_t(S_i)$, all~$u_l,v_l$ lie in~$(X_i\cup \dot{X_i})^{\ast}$, and we have~$u_l\overline{f_l}v_l\leq w_1<w$. (For instance, say $(f,g)_w=x_1\dot{x_2}x_3-\dot{x_3}$,~$f_1=\dot{x_2}x_3-\dot{x_1}\in \Psi_t(S_i)$ and $f_2=x_1\dot{x_1}-\dot{x_3}\in \Psi_t(S_i)$, then we have~$(f,g)_w=x_1f_1+f_2.$)

By Lemma~\ref{emb},  it follows that~$\mathsf{V}^{(t)}(X_i|S_i)$ embeds into~$\mathcal{A}_{i,t}^{(t)}=k\langle X_i\cup \dot{X_i}|\Psi_t(S_i)\cup R_{i,t}\rangle^{(t)}$. Since~$\{x+\mathsf{Id}^{(t)}(S_i)\mid x\in X_i\}$ forms a linear basis for~$\mathsf{V}^{(t)}(X_i|S_i)$, we obtain that its image
$\{\dot{x}+\mathsf{Id}(\Psi_t(S_i)\cup R_{i,t})\mid \dot{x}\in \dot{X}_i\}$ is linear independent,
and thus every coefficient~$\delta_n$ is zero.   Therefore, the set~$\Psi_t(S_i)\cup R_{i,t}$ forms a Gr\"obner-Shirshov basis for~$\mathcal{A}_{i,t}$.
\end{proof}

\begin{coro} \label{coro-gsb-at}
   With the notations in Remark~\ref{nota-relation}, the set~$\cup_{i\in I}(\Psi_t(S_i)\cup R_{i,t})$ forms a Gr\"obner-Shirshov basis for~$\mathcal{A}_t$.
\end{coro}

\begin{proof}
  By Lemmas~\ref{lemm-gsb} and~\ref{ass-gsb}, the result follows immediately.
\end{proof}

Now we can establish the first main result of the article, in which we construct a linear basis for a free product of arbitrary trialgebras (resp. dialgebras).

\begin{thm}\label{nf}
Let the notations be as in Remark~\ref{nota-relation}.
  Let $$Y^{(t)}=\{y_1...y_n \mid  y_1\wdots y_n\in \cup_{i\in I}(X_i'\cup\dot{X_i}), \mathsf{ind}(y_p)\neq \mathsf{ind}(y_{p+1}), n\in \mathbb{N}^{+}, 1\leq p\leq n-1\}.$$ Then   the set $$\{\Psi_{t}^{-1}(u)+\mathsf{Id}^{(t)}(\cup_{i\in I}S_i) \mid u\in Y^{(t)}\cap \Psi_{t}(\mathsf{V}^{(t)}(\cup_{i\in I}X_i))\}$$ forms a linear basis for the free product~$\mathsf{V}^{(t)}(\cup_{i\in I}X_i|\cup_{i\in I}S_i)$ of $\{\mathsf{V}^{(t)}(X_i| S_i)\mid i\in I\}$, $t\in \{2,3\}$.
\end{thm}

\begin{proof}  Denote~$\cup_{i\in I}(\Psi_t(S_i)\cup R_{i,t})$ by~$S$, and denote~$\cup_{i\in I}(\dot{X_i}\cup X_i)$ by~$X$. Then  by Lemma~\ref{l1} and Corollary~\ref{coro-gsb-at}, the set
$$\mathsf{Irr}(S)=\{u+\mathsf{Id}(S)\mid u\in  X^*\setminus X^*\overline{S}X^*\}$$
   forms a set of  linear basis for the quotient associative algebra $ \mathcal{A}_t$.
   By Lemma~\ref{emb}, $\mathsf{V}^{(t)}( \cup_{i\in I}(\dot{X_i}\cup X_i) |\cup_{i\in I}(\Psi_t(S_i)\cup R_{i,t}))$   has a $k$-basis:
$$ \{\Psi_{t}^{-1}(u)+\mathsf{Id}^{(t)}(S) \mid u\in (X^*\setminus X^*\overline{S}X^*)\cap \Psi_{t}(\mathsf{V}^{(t)}(\cup_{i\in I}X_i))\}.$$
So it suffices to show
   $$Y^{(t)}\cap \Psi_{t}(\mathsf{V}^{(t)}(\cup_{i\in I}X_i)) =(X^*\setminus X^*\overline{S}X^*) \cap \Psi_{t}(\mathsf{V}^{(t)}(\cup_{i\in I}X_i)).$$

For each $i\in I$, since $\{x+\mathsf{Id}^{(t)}(S_i)\mid x\in X_i\}$ is a linear basis for~$\mathsf{V}^{(t)}(X_i|S_i)$ and $S_i$ is the multiplication table for~$\mathsf{V}^{(t)}(X_i|S_i)$,  we have
\begin{equation}\label{eq-lm}
\overline{\Psi_t(S_i)}=
\begin{cases}
\{yz\mid y,z\in X_i\cup\dot{X_i}, \text{ and at least one of } y,z \text{ lies in } \dot{X_i}\} \text{ if }t=3,\\
\{yz\mid y,z\in X_i\cup\dot{X_i}, \text{ and exactly one of } y,z \text{ lies in } \dot{X_i}\} \text{ if }t=2,
\end{cases}
\end{equation}
and
 $$\overline{\varphi(\Psi_t(S_i))}=\{yz\mid y,z\in X_i \}\subseteq \overline{R_{i,t}}.$$
 Moreover, for every~$i\in I$ and for all~$y,y'\in \dot{X_i}$, we have
\begin{equation}\label{eqyy}
  yy'\in \overline{\Psi_3(S_i)} \mbox{ if } t=3, \mbox{ and } yy'\notin \Psi_{2}(\mathsf{V}^{(2)}(\cup_{i\in I}X_i)) \mbox{ if } t=2.
\end{equation}
Therefore, if  $u=y_1...y_n\in (X^*\setminus X^*\overline{S}X^*) \cap \Psi_{t}(\mathsf{V}^{(t)}(\cup_{i\in I}X_i))$, where $y_1\wdots y_n$ lies in~$X$, then we have
$$\mathsf{ind}(y_p)\neq \mathsf{ind}(y_{p+1})$$
for every integer~$p$ such that~$ 1\leq p\leq n-1$.
Since $X_i'=X_i^*\setminus X_i^*\overline{R_{i,t}}X_i^*=X_i\setminus\{\bar{f}\mid f\in R_{i,t}\}$, we deduce that $$y_1\wdots y_n\in \cup_{i\in I}(X_i'\cup\dot{X_i} ).$$
It follows immediately that
$$ (X^*\setminus X^*\overline{S}X^*) \cap \Psi_{t}(\mathsf{V}^{(t)}(\cup_{i\in I}X_i))\subseteq Y^{(t)},$$
and thus we obtain
$$ (X^*\setminus X^*\overline{S}X^*) \cap \Psi_{t}(\mathsf{V}^{(t)}(\cup_{i\in I}X_i))\subseteq Y^{(t)} \cap \Psi_{t}(\mathsf{V}^{(t)}(\cup_{i\in I}X_i)).$$

On the other hand,   since~$\overline{R_{i,t}}$ is a subset of~$X_i^*$,    it is clear that every~$\dot{x}\in X$ is not a leading monomial of any element in~$\cup_{i\in I}(\Psi_t(S_i)\cup R_{i,t})$.
Moreover, since~$X_i'= X_i\setminus\{\bar{f}\mid f\in R_{i,t}\}$ and since the leading monomial of every element in~$\Psi_t(S_i)$ for every~$i\in I$ has length at least two, we deduce that every $x\in X_i'$ is not a leading monomial of any element in~$\cup_{i\in I}(\Psi_t(S_i)\cup R_{i,t})$.  Combining the above reasoning with~\eqref{eq-lm} and~\eqref{eqyy}, we deduce that
$$
Y^{(t)}\cap \Psi_{t}(\mathsf{V}^{(t)}(\cup_{i\in I}X_i))\subseteq (X^*\setminus X^*\overline{S}X^*) \cap \Psi_{t}(\mathsf{V}^{(t)}(\cup_{i\in I}X_i)). $$
 The result follows immediately.
\end{proof}

By Lemma~\ref{emb} and by Theorem~\ref{nf}, we immediately obtain the following alternative description of a free product of arbitrary trialgebras (resp. dialgebras).

\begin{coro}\label{co-nf}
  Let the notations be as in Remark~\ref{nota-relation}.  Then the subalgebra $\mathcal{P}_t$ of
$$\mathcal{A}_t^{(t)}=k\langle\cup_{i\in I}(X_i\cup\dot{X_i})|\cup_{i\in I}(\Psi_t(S_i)\cup \varphi(\Psi_t(S_i))) \rangle^{(t)}$$
 generated by~  $\{\dot{x}+\mathsf{Id}(\cup_{i\in I}(\Psi_t(S_i)\cup R_{i,t}))\mid \dot{x}\in \cup_{i\in I}\dot{X_i}\}$  is the free product of algebras  $\{\mathsf{V}^{(t)}(X_i|S_i)\mid i\in I \}$.  Moreover,  with the notations of Theorem~\ref{nf},
the set
$$\{u+\mathsf{Id}(\cup_{i\in I}(\Psi_t(S_i)\cup R_{i,t})) \mid u\in Y^{(t)}\cap \Psi_{t}(\mathsf{V}^{(t)}(\cup_{i\in I}X_i))\}$$ forms a linear basis for~$\mathcal{P}_t$, $t\in \{2,3\}$.
\end{coro}

 Let $(T, \vdash, \dashv,\perp)$ (resp. $(D, \vdash, \dashv)$) be an arbitrary trialgebra (resp. dialgebra) and let~$U^{(3)}$ (resp. $U^{(2)}$) be the ideal of~$T$ (resp. $D$) generated by~
$$
\{f\vdash g-f\dashv g, f\vdash g-f\perp g\mid f,g \in T\}\ \ \ (resp. \ \{f\vdash g-f\dashv g\mid f,g\in D\}).
$$
Then $T/U^{(3)}$ (resp. $D/U^{(2)}$) is called the \emph{associated associative algebra} of the  trialgebra $T$ (resp. dialgebra $D$).

For an arbitrary associative algebra~$\mathcal{B}$, if we define~$f\vdash g=f\dashv g=f\perp g=fg$ for all~$f,g\in \mathcal{B}$, then~$(\mathcal{B}, \vdash, \dashv,\perp)$ is a trialgebra which is denoted by~$\mathcal{B}^{[3]}$, and $(\mathcal{B}, \vdash, \dashv)$ is a dialgebra which is denoted by~$\mathcal{B}^{[2]}$ (cf. Equations~\eqref{talg} and~\eqref{dalg}).
 Now we show that we can identify~$k\langle X_i |\varphi(\Psi_t(S_i))\rangle$  with the associated associative algebra of the  trialgebra~$\mathsf{V}^{(t)}(X_i|S_i)$ when~$t=3$ (resp. dialgebra when~$t=2$).

\begin{lemm}\label{ass-ass}
  With the notations in Remark~\ref{nota-relation}, let $U^{(3)}$ \emph{(resp. $U^{(2)}$)} be the ideal of~$\mathsf{V}^{(3)}(X_i)$ \emph{(resp. $\mathsf{V}^{(2)}(X_i)$)} generated by~$\{f\vdash g-f\dashv g, f\vdash g-f\perp g\mid f,g \in \mathsf{V}^{(3)}(X_i)\}$
  \emph{(resp. $\{f\vdash g-f\dashv g\mid f,g\in \mathsf{V}^{(2)}(X_i)\}$)}.
 Then~$\mathsf{V}^{(t)}(X_i|U^{(t)}\cup S_i )$ is isomorphic to~$k\langle X_i|\varphi(\Psi_t(S_i))\rangle^{[t]}$,  for~$t\in \{2,3\}$. Moreover, the set~$\{x+\mathsf{Id}^{(t)}(U^{(t)}\cup S_i)\mid x\in X_i'\}$ forms a linear basis for~$\mathsf{V}^{(t)}(X_i|U^{(t)}\cup S_i )$.
\end{lemm}
\begin{proof}
First of all, it is easy to see that all the operations in~$\mathsf{V}^{(t)}(X_i|U^{(t)}\cup S_i )$ coincide and then $\mathsf{V}^{(t)}(X_i|U^{(t)}\cup S_i )$ is also an associative algebra.  Since~$k\langle X_i \rangle$ is the free associative algebra generated by~$X_i$, we can define an (associative) algebra homomorphism~$\psi$ from~$k\langle X_i \rangle$ to~$\mathsf{V}^{(t)}(X_i|U^{(t)}\cup S_i )$ by the rule~$\psi(x)=x+\mathsf{Id}^{(t)}(U^{(t)}\cup S_i)$ for every~$x\in X_i$.    Moreover, since the operations~$\vdash$ and~$\dashv$ in~$\mathsf{V}^{(t)}(X_i|U^{(t)}\cup S_i )$ coincide, for every element~$f\in S_i$, we have
$$\psi(\varphi(\Psi_t(f)))+\mathsf{Id}^{(t)}(U^{(t)}\cup S_i)=f+\mathsf{Id}^{(t)}(U^{(t)}\cup S_i)=\mathsf{Id}^{(t)}(U^{(t)}\cup S_i),$$
in other words,  $\psi$ induces an (associative) algebra homomorphism, denoted by~$\widetilde{\psi}$, from the algebra~$k\langle X_i|\varphi(\Psi_t(S_i))\rangle$ to~$\mathsf{V}^{(t)}(X_i|U^{(t)}\cup S_i )$ by the rule:
$$\widetilde{\psi}(x+\mathsf{Id}(\varphi(\Psi_t(S_i))))=
x+\mathsf{Id}^{(t)}(U^{(t)}\cup S_i)$$ for every~$x\in X_i$. Clearly, $\widetilde{\psi}$ is also a dialgebra (resp. trialgebra) homomorphism when~$t=2$ (resp. $t=3$) from~$k\langle X_i|\varphi(\Psi_t(S_i))\rangle^{[t]}$ to~$\mathsf{V}^{(t)}(X_i|U^{(t)}\cup S_i )$.

On the other hand, we have an algebra homomorphism
 $$
 \widetilde{\varphi}:  \ \mathsf{V}^{(t)}(X_i|U^{(t)}\cup S_i ) \longrightarrow k\langle X_i|\varphi(\Psi_t(S_i))\rangle^{[t]},\ x+\mathsf{Id}^{(t)}(U^{(t)}\cup S_i)\mapsto x+\mathsf{Id}(\varphi(\Psi_t(S_i)))
 $$
 for every~$x\in X_i$  (as a dialgebra homomorphism when~$t=2$ and as a trialgebra homomorphism when~$t=3$). Moreover, $\widetilde{\varphi}$ is clearly an associative algebra homomorphism.

Obviously, $\widetilde{\varphi}$ is the inverse of~$\widetilde{\psi}$. The first claim follows immediately.

As for the last assertion, by Remark~\ref{nota-relation}, the set~$ \{x+\mathsf{Id}(\varphi(\Psi_t(S_i)))\mid x\in X_i'\}$ forms a linear basis for~$k\langle X_i|\varphi(\Psi_t(S_i))\rangle$, so the set~$\{x+\mathsf{Id}^{(t)}(U^{(t)}\cup S_i)\mid x\in X_i'\}$ forms a linear basis for~$\mathsf{V}^{(t)}(X_i|U^{(t)}\cup S_i )$.
\end{proof}

 We conclude the article with an observation on a free product of arbitrary trioids. Assume that~$\{(X_i,\vdash, \dashv, \perp)\mid i\in I\}$ is a family of disjoint trioids. (Sometimes  we use the notation~$X_i$ both for the trioid~$(X_i,\vdash, \dashv, \perp)$ and for its underlying set to simplify the notations if no confusions arise.)   Let~$\mathsf{Vid}^{(3)}(X_i)$ be the free trioid generated by the set~$X_i$  and let~$\rho_i$ be the congruence of~$\mathsf{Vid}^{(3)}(X_i)$ generated by
\begin{equation}\label{trioid-si}
 S_i=\{(x\delta y,\{x\delta y\})\mid x,y\in X_i, \delta\in\{\vdash, \dashv, \perp\}\},
\end{equation}
where~$\{x\delta y\}$ is the letter in the set~$X_i$ such that we have~$x\delta y=\{x\delta y\}$ in~$(X_i,\vdash, \dashv, \perp)$.
 It is clear that we have
$$
(X_i,\vdash, \dashv, \perp) \cong \mathsf{Vid}^{(3)}(X_i)/\rho_i.
$$  Denote by~$kX_i$ the linear space with a linear basis~$X_i$, and extend the operations~$\vdash,\dashv,\perp$ on~$kX_i$ in a natural way.  Then we have the trialgebra $(kX_i,\vdash,\dashv,\perp)$ and an isomorphism of trialgebras
 $$(kX_i,\vdash,\dashv,\perp)\cong \mathsf{V}^{(3)}(X_i|S_i), $$
where we identify the set $S_i$ with the set
$\{x\delta y-\{x\delta y\}\mid x,y\in X_i, \delta\in\{\vdash, \dashv, \perp\}\}.$

Clearly, $(X_i,\vdash,\dashv,\perp)$ is isomorphic to the subtrioid of~$\mathsf{V}^{(3)}(X_i|S_i)$ generated by the set~$\{x+\mathsf{Id}^{(3)}(S_i)\mid x\in X_i\}$.
Therefore, by Lemma~\ref{emb}, $(X_i,\vdash,\dashv,\perp)$ is isomorphic to the subtrioid of $k\langle X_i\cup\dot{X_i}|\Psi_{3}(S_i)\cup \varphi(\Psi_{3}(S_i)) \rangle^{(3)}$ generated by the set
$$\dot{X_i}+\mathsf{Id}(\Psi_{3}(S_i)\cup \varphi(\Psi_{3}(S_i))):= \{\dot{x}+\mathsf{Id}(\Psi_{3}(S_i)\cup \varphi(\Psi_{3}(S_i)))\mid \dot{x}\in \dot{X}_i\}. $$
Finally,  by Lemmas~\ref{thm-fpt} and~\ref{emb},  the subtrioid of~$k\langle\cup_{i\in I}(X_i\cup\dot{X_i})|\cup_{i\in I}(\Psi_3(S_i)\cup \varphi(\Psi_3(S_i))) \rangle^{(3)}$ generated~$\cup_{i\in I}(\dot{X_i}+\mathsf{Id}(\Psi_{3}(S_i)\cup \varphi(\Psi_{3}(S_i)))$ is the free product~$\mathsf{Vid}^{(3)}(\cup_{i\in I}X_i|\cup_{i\in I}S_i)$ of the family of trioids~$\{(X_i,\vdash,\dashv,\perp)\mid i\in I\}=\{\mathsf{Vid}^{(3)}(X_i|S_i)\mid i\in I\}$.

Denote by~$\rho_i'$ the congruence of~$(X_i,\vdash,\dashv,\perp)$ generated by
$$\{(x\delta y,x\delta' y) \mid x,y\in X_i, \delta,\delta'\in\{\vdash, \dashv, \perp\}\}$$
and define
$$(B_i,\vdash,\dashv,\perp):=(X_i,\vdash,\dashv,\perp)/\rho_i'.$$
 Suppose that~$R_i$  is  a Gr\"obner-Shirshov basis of~$k\langle X_i |  \varphi(\Psi_3(S_i))\rangle$ containing~$\varphi(\Psi_3(S_i))$,  and define
$$X_i'=X_i\setminus \{\bar{f} \mid f\in R_i \}. $$
Then  by Lemma \ref{ass-ass},  we deduce that
$$X_i'\rho_i':=\{x\rho_i'\mid x\in X_i'\}$$ forms  a set of normal forms of elements in~$(B_i,\vdash,\dashv,\perp)$.
Therefore, we have
$$(B_i,\vdash,\dashv,\perp)=(X_i,\vdash,\dashv,\perp)/\rho_i'=(\{x\rho_i' \mid x\in X_i\},\vdash,\dashv,\perp)=(\{x\rho_i' \mid x\in X_i'\},\vdash,\dashv,\perp).$$
 Define
$$V_3=\mathsf{V}^{(3)}(\cup_{i\in I} X_i), \ Y_3=\cup_{i\in I}(\dot{X_i}\cup X_i') ,$$ and define
\begin{equation*}
\mathsf{FP}(X_i)_{i\in I}=\{y_1...y_n\in \Psi_3(V_3) \mid
y_1\wdots y_n\in Y_3,n\in\mathbb{N}^{+},
\mathsf{ind}(y_p)\neq \mathsf{ind}(y_{p+1}), p\leq n-1\}.
\end{equation*}
Now we define operations~$\vdash', \dashv', \perp'$ on~$\mathsf{FP}(X_i)_{i\in I}$
  such that~$(\mathsf{FP}(X_i)_{i\in I}, \vdash', \dashv', \perp')$ becomes  the free product of~$\{(X_i, ,\vdash,\dashv,\perp)\mid i\in I\}$.

For all letters~$y, z\in Y_3$, define~$\mathsf{Red}(yz)$ to be the word~$u$ in the free semigroup~$Y_3^+$ such that~$u+\mathsf{Id}(\cup_{i\in I}(\Psi_3(S_i)\cup \varphi(\Psi_3(S_i)))$ is the   normal form   of~$yz+\mathsf{Id}(\cup_{i\in I}(\Psi_3(S_i)\cup \varphi(\Psi_3(S_i)))$  obtained in the associative algebra~$k\langle \cup_{i\in I}(X_i\cup \dot{X_i})|\cup_{i\in I}(\Psi_3(S_i)\cup \varphi(\Psi_3(S_i))\rangle $  by applying the Gr\"onber-Shirshov basis~$\cup_{i\in I}(\Psi_3(S_i)\cup R_i)$.
There exist several cases:

\ITEM1 If~$\mathsf{ind}(y)\neq \mathsf{ind}(z)$, then we have~$\mathsf{Red}(yz)=yz$;

\ITEM2 If $y,z\in X_i'$  for some index $i\in I$, then we can rewrite $yz$ into a letter in~$X_i'$ by relations in $R_i$, namely, we have~$(y\vdash z)\rho_i'=\mathsf{Red}(yz)\rho_i'$ in~$(B_i,\vdash,\dashv, \perp)$. (Note that in the trioid~$(B_i,\vdash,\dashv, \perp)$, the operations~$\vdash,\dashv, \perp$ coincide.)

\ITEM3 If $y,z\in \dot{X_i}\cup X_i'$ for some index $i\in I$ and~$\{y,z\}\cap \dot{X_i}\neq \emptyset$, then we can rewrite~$yz$ into a letter in~$\dot{X_i}$  by relations in~$\Psi_3(S_i)$.  More precisely, if~$y=\dot{a}\in \dot{X_i}$ and~$z\in X_i'$, then $\mathsf{Red}(yz)=\dot{b}$, where~$b$ is the letter in~$X_i$ such that we have~$b=a\dashv z$ in~$(X_i,\vdash, \dashv, \perp)$; if~$y\in X_i'$ and~$z=\dot{a}\in \dot{X_i}$, then $\mathsf{Red}(yz)=\dot{c}$, where~$c$ is the letter in~$X_i$ such that  we have~$c=y\vdash a$ in~$(X_i,\vdash, \dashv, \perp)$; if~$y=\dot{a}$ and~$z=\dot{b}$, then  $\mathsf{Red}(yz)=\dot{d}$, where~$d$ is the letter in~$X_i$ such that  we have~$d=a\perp b$ in~$(X_i,\vdash, \dashv, \perp)$.

By Corollary~\ref{co-nf},  we  obtain the main result of this article, in which we construct free products of arbitrary (disjoint) trioids.

\begin{thm}\label{nffreeproduct}
Let~$\{(X_i,\vdash,\dashv, \perp)\mid i\in I\}$ be a family of disjoint trioids.  Then~$(\mathsf{FP}(X_i)_{i\in I},\vdash', \dashv',\perp')$ is the free product of~$\{X_i\mid i\in I\}$ with the following products:
\begin{align*}
&y_1...y_n\vdash' z_1...z_m= \varphi(y_1)...\varphi(y_{n-1}) \mathsf{Red}(\varphi(y_n)z_1)z_2...z_m,\\
&y_1...y_n\dashv' z_1...z_m=y_1...y_{n-1}\mathsf{Red}(y_n\varphi(z_1))\varphi(z_2)...\varphi(z_m),\\
&y_1...y_n\perp' z_1...z_m=y_1...y_{n-1}\mathsf{Red}(y_nz_1)z_2...z_m,
\end{align*}
for all $y_1...y_n,z_1...z_m\in \mathsf{FP}(X_i)_{i\in I}$.
\end{thm}

\begin{proof}
  Denote~$\cup_{i\in I}(\Psi_t(S_i)\cup \varphi(\Psi_t(S_i)))$ by~$S$, and denote~$\cup_{i\in I}\dot{X_i}$ by~$\dot{X}$. Obviously, $(\mathsf{FP}(X_i)_{i\in I},\vdash', \dashv',\perp')$ is isomorphic to the sub-trioid of~$\mathcal{A}_3:=k\langle\cup_{i\in I}(X_i\cup\dot{X_i})|S \rangle $ generated by $\dot{X}+\mathsf{Id}(S):=\{\dot{x}+\mathsf{Id}(S)\mid \dot{x}\in \dot{X}\}$, which is isomorphic to~$\mathsf{Vid}^{(3)}(\cup_{i\in I}X_i|\cup_{i\in I}S_i)$. The result follows immediately.
\end{proof}

\begin{rema}
 Similar to Theorem \ref{nffreeproduct}, one can obtain a set of normal forms for free product of arbitrary dimonoids and one can deduce the multiplication table for the free product, which turns out to be essentially the same as that constructed in~\cite{Zhuchok13}.
\end{rema}

\end{document}